\documentclass[a4paper,12pt]{article}

\usepackage{amsmath,amsthm,amssymb,cite,geometry, graphicx,url}

\newtheorem{theorem}{Theorem}

\numberwithin{theorem}{section}
\usepackage{color}
\newtheorem{conjecture}[theorem]{Conjecture}
\newtheorem{corollary}[theorem]{Corollary}
\newtheorem{lemma}[theorem]{Lemma}

\newtheorem{proposition}[theorem]{Proposition}
\newtheorem{remark}[theorem]{Remark}

\theoremstyle{definition}

\newtheorem{definition}[theorem]{Definition}
\newtheorem{example}[theorem]{Example}

\newcommand{\R}{\mathbb{R}}

\newcommand{\N}{\mathbb{N}}

\DeclareMathOperator{\aff}{aff}

\DeclareMathOperator{\cone}{cone}
\DeclareMathOperator{\co}{co}
\DeclareMathOperator*{\dbigcup}{\dot \bigcup}

\DeclareMathOperator{\lin}{lin}

\DeclareMathOperator{\qri}{qri}
\DeclareMathOperator{\icr}{icr}

\DeclareMathOperator{\ri}{ri}
\DeclareMathOperator{\fri}{fri}

\DeclareMathOperator{\qi}{qi}
\newcommand{\proplhd}{%
  \mathrel{\ooalign{$\lneq$\cr\raise.22ex\hbox{$\lhd$}\cr}}}

\title{Face relative interior of convex sets in topological vector spaces}
\author{R. D\'iaz Mill\'an\thanks{Deakin University, Waurn Ponds, Australia} { and} Vera Roshchina\thanks{UNSW Sydney, Australia}}

\begin{document}

\maketitle

\begin{abstract}
A new notion of face relative interior for convex sets in topological real vector spaces is introduced in this work. Face relative interior is grounded in the facial structure, and may capture the geometry of convex sets in topological vector spaces better than other generalisations of relative interior. 

We show that the face relative interior partitions convex sets into face relative interiors of their closure-equivalent faces (different to the partition generated by intrinsic cores), establish the conditions for nonemptiness of this new notion, compare the face relative interior with other concepts of convex interior and prove basic calculus rules.
\end{abstract}

\section{Introduction}

There are several notions generalising the relative interior of a convex set in the Euclidean setting to real vector spaces. Perhaps the most well-known is the purely algebraic notion of intrinsic core defined for convex sets in arbitrary real vector spaces and the quasi-relative interior defined for convex sets in topological vector spaces. The intrinsic core has a deep and beautiful connection with the facial structure of convex sets, however, it may be empty even for `very large' sets (see our recent review \cite{mill-rosh} and references therein). In contrast to this, quasi-relative interior doesn't seem to have a neat interpretation in terms of faces, but its nonemptiness in practically useful settings (e.g. separable Banach spaces) appears to be more important for applications.

We argue that an alternative to the quasi-relative interior can be defined intrinsically in terms of faces while being nonempty under similar conditions. We call this new notion face relative interior. While it remains to be seen how useful this new generalisation may be for applications, some of the properties of face relative interior (that other notions do not satisfy) suggest that it may be a more natural choice than the quasi-relative interior in some settings, especially when one is concerned with the facial structure of convex sets in general topological vector spaces.

The \emph{face relative interior} of a convex set consists of all points for which the topological closure of their minimal face contains the entirety of this convex set. The face relative interior is sandwiched between the intrinsic core and the quasi-relative interior, and is different to both. We devote Section~\ref{sec:facereldef} to the proof of this fact and provide several examples that demonstrate the differences between the face-relative interior and other notions. 

Convex sets in separable Banach spaces have nonempty face-relative interiors (Corollary~\ref{cor:friBanach}). More generally, nonemptiness is guaranteed under general assumptions, similar to the well-known conditions for the nonemptiness of the quasi-relative interior. We discuss this in Section~\ref{sec:nonemptiness}, where we also systematise the results on a key notion of CS-closed sets in topological vector spaces that are otherwise scattered in the literature. 

Section~\ref{sec:calculus} is dedicated to exploring the calculus of face relative interiors, which resembles the calculus of quasi-relative interiors. Interestingly, we didn't need to invoke duality in our proofs.

Finally, in Section~\ref{sec:equiv} we argue that the face relative interior aligns more naturally with the facial structure of the convex set, compared to the quasi-relative interior. While any convex set is partitioned as the disjoint union of the intrinsic cores of its faces (see \cite{mill-rosh}), this is not true for the quasi-relative interior. However identifying faces that have the same closure we obtain a similar disjoint partition via the face relative interiors, based on the observation that faces that aren't closure-equivalent have non-intersecting face relative interiors. We show that this structural property is not available for the quasi-relative interiors: it may happen that faces with different closures have overlapping quasi-relative interiors (see Example~\ref{nonpartition}). We also show that the partition via face-relative interiors of faces is different to the partition via the intrinsic cores using explicit examples. 

We finish the paper with conclusions, where we also review some open questions and topics for future investigation.

\section{Face relative interior: definition and examples}\label{sec:facereldef}

The definition of face relative interior is based on the facial structure of the convex set. For a self-contained exposition, we briefly review the relevant ideas first.  We begin with a discussion on the facial structure of convex sets in real vector spaces, recap several useful results and definitions related to minimal faces, and then move on to defining the face relative interior, placing it in the context of other notions, and working out the conditions for its nonemptiness. 

\subsection{Notions of relative interior}

Faces are the main structural elements of convex sets: faces of a convex set form a lattice with respect to the set inclusion, and the facial structure is key for studying the geometry of convex sets.

For now, assume that $X$ is an arbitrary real vector space and $\N:=\{1,2,\dots\}$ denote the set of the natural numbers. Let $C\subseteq X$ be a convex set. A convex subset $F\subseteq C$ is called a {\em face} of $C$ if for every $x\in F$ and every $y,z \in C$ such that $x\in (y,z)$, we have $y,z\in F$.

The empty set is a face of $C$, and the set $C$ itself is its own face. Nonempty faces that don't coincide with $C$ are called proper. It is customary to write $F\unlhd C$ for a face $F$ of $C$. 

Given a subset $S$ of a convex set $C$, the smallest face, with respect to the inclusion, $F\unlhd C$ containing the set $S$, is called the minimal face of $S$ in $C$; in other words, $F$ is the intersection of all faces of $C$ that contain $S$. The minimal face is well-defined since any intersection of faces of a convex set is also a face. We denote the minimal face of a convex set $C$ containing $S\subset C$ by  $F_{\min}(S,C)$. When $S$ is a singleton ($S=\{x\}$ for some $x\in X$), we abuse the notation and write $F_{\min}(x,C)$ meaning $F_{\min}(\{x\},C)$. 

It was shown in \cite{mill-rosh} that the minimal face of a convex set $C$ containing $x\in X$ satisfies  
\begin{equation}\label{eq:frcharlin}
F_{\min} (x,C) = C\cap (\lin \cone (C-x)+x),
\end{equation}
where by $\lin K$ we denote the lineality space of a convex cone $K$: the largest linear subspace contained in $K$, and $\cone $ denotes the conic hull (for a convex set $C$ we have $\cone C = \R_{+} C = \{\alpha x\,|\, x\in C, \alpha \geq 0\}$). We abuse the notation and write $C+x$ for $C+\{x\}$, the Minkowski sum of the set $C$ and the singleton $\{x\}$. 

The intrinsic core (also known as pseudo-relative interior and the set of inner points) was studied extensively in the 50's (see our recent review \cite{mill-rosh} for more information). The intrinsic core  $\icr C$ of a convex set $C$ is a subset of $C$ such that $x\in \icr C$ if and only if for every $y\in C$ there exists $z\in C$ with $x\in (y,z)$.\label{page:deficr} 
One can also define the intrinsic core via minimal faces, \begin{equation}\label{eq:icrminface}
\icr C=\{x\in C\,|\, F_{\min}(x,C)=C\},
\end{equation}
and using the characterisation \eqref{eq:frcharlin}, 
\begin{equation}\label{eq:icrcone}
\icr C=\{x\in C\,|\, \cone(C-x) \text{ is a linear subspace} \}.
\end{equation}

The quasi-relative interior was introduced by Zarantonello in 1971 under the name of inner points, see \cite{Zarantonello1971ProjectionsOC}, but further developed and popularised by Borwein and  Lewis in 1992 \cite{BorweinLewisPartiallyFinite}. We also note that it was independently introduced in \cite{HadjisavvasSchaible} under the name of inner points. We would like to thank our referee for pointing out this important historical perspective.

Quasi-relative interior is a fairly well-studied notion, with many theoretical and practical applications. For instance, it was used in \cite{Dontchev} to formulate a Slater-like condition for the cone of nonnegative functions in $L^2$; in \cite{WeakEfficiency} the notion of quasi-relative interior is used to obtain necessary and sufficient conditions for the existence of saddle points of a Lagrangian function in a class of vector optimisation problems; in \cite{SetValuedSystems} the relation between linear separation and saddle points of Lagrangian functions in the context of variational inequalities is studied via the quasi-relative interior. Some other applications are discussed in \cite{InfDim, RemarksInf,Lagrange,Regularity,Subconvex,EfficientCoderivatives,MR3297972,MR3735852}. A detailed overview of Borwein and Lewis' work with examples and additional results is given in \cite{Lindstrom}. Zalinescu \cite{ZalinescuThreePb, ZalinescuOnTheUse} studied duality and separation in the context of quasi-relative interior and resolved a number of open questions related to the quasi-relative interior.  

Following \cite[Definition 2.6]{BorweinGoebel}, we define the {\it quasi-relative interior} $\qri C$ of a convex set $C\subseteq X$, where $X$ is a topological vector space, as 
\[
\qri C=\{x\in C: \overline{\cone}(C-x) \text{ is a linear subspace} \}.
\]

Some of the key advantages of quasi-relative interior over the intrinsic core is that quasi-relative interior is guaranteed to be nonempty under mild conditions (we discuss these in detail in Section~\ref{sec:nonemptiness}), and it has favourable properties when working with dual objects.  

Finally, we mention another generalisation of the relative interior, which is commonly referred to as the actual `relative interior', denoted by $\ri C$. This is the interior of $C$ relative to $\overline{\aff C}$, the closed affine hull of $C$. 

We are now ready to define the new notion of face relative interior. 

\begin{definition}[Face relative interior]\label{def:fri}
Let $C\subseteq X$ be a convex subset of a topological vector space $X$. We define the {\it face relative interior} of $C$ as
$$\fri C:=\{x\in C\,|\, C\subseteq \overline{F_{\min}(x,C)}\}.$$
\end{definition}

Our definition is analogous to the definition of the quasi-relative interior in that it also incorporates a topological closure. Curiously, all three key topological generalisations of the relative interior can be obtained by taking closures of appropriate objects, using three different but equivalent characterisations of the intrinsic core. The characterisations \eqref{eq:icrminface} and \eqref{eq:icrcone} correspond to face relative and quasi-relative interiors. Since the intrinsic core of a convex set $C$ can be seen as the core of $C$ with respect to its affine hull, relative interior can be seen as the `topological closure' of this characterisation. 

We will focus on the properties and calculus of the face relative interior in detail in Section~\ref{sec:calculus}. For now we note that the face relative interior is a convex set.

\begin{proposition}\label{prop:friconvex} Let $C$ be a convex subset of a topological vector space $X$. Then $\fri C$ is a convex subset of $C$.
\end{proposition}
\begin{proof}
Suppose that $x,y\in\fri C$ and $z\in (x,y)$. Since $C$ is convex, $z\in C$. Since $z\in (x,y)$, we must have $x,y\in F_{\min}(z,C)$ by the definition of a face. Since $x\in \fri C$, it follows that 
\[
C \subseteq \overline{F_{\min}(x,C)} \subseteq \overline{F_{\min}(z,C)},
\]
proving that $z \in \fri C$. Then $\fri C$ is a convex subset of $C$.
\end{proof}

\subsection{The sandwich theorem}

The face relative interior is sandwiched between the intrinsic core and quasi-relative interior, as we show next in Theorem~\ref{thm:interiorsandwich}. We prove this theorem first and then proceed with examples that show face relative interior to be different from the other three notions that we discussed (intrinsic core, quasi-relative and relative interiors). 

\begin{theorem}[Sandwich]\label{thm:interiorsandwich} Let $C$ be a convex subset of a topological vector space $X$. Then 
\begin{equation}\label{eq:mainsandwich}
\ri C \subseteq \icr C \subseteq \fri C \subseteq \qri C.
\end{equation} 
\end{theorem}
\begin{proof} The first inclusion $\ri C\subseteq \icr C$ is well-known (e.g. see \cite[Theorem~2.12]{BorweinGoebel}). 

Suppose that $C\subseteq X$ is convex, and let $x\in \icr C$. By \eqref{eq:icrminface} we have $F_{\min }(x,C) = C$, and hence 
\[
C \subseteq \overline{F_{\min}(x,C)}.
\]
By the definition of face relative interior, we conclude that $x\in \fri C$. We have therefore shown that $\icr C \subseteq \fri C$. 

It remains to demonstrate that $\fri C \subseteq \qri C$. Let $x\in \fri C$. Then 
\[
C \subseteq \overline{F_{\min}(x,C)}.
\]

From \eqref{eq:frcharlin} we have 
\begin{equation}\label{eq:345566}
C \subseteq \overline{F_{\min}(x,C)} = \overline {C\cap (\lin \cone (C-x)+x)}\subseteq  \overline {\lin \cone (C-x)+x}.
\end{equation}
In a topological vector space, the topology is translation invariant (see \cite[Section~5.1]{hitchhiker}), for any $A\subset X$ and $x\in X$ we have $\overline{A+x}{\color{blue}=} \overline{A}+x$.

Then we have that $ \overline {\lin \cone (C-x)+x}{\color{blue}=}\overline {\lin \cone (C-x)}+x$, so subtracting $x$ from both sides of the inclusion \eqref{eq:345566} we have
\begin{equation}\label{eq:technical045}
C - x\subseteq \overline {\lin \cone (C-x)},    
\end{equation}
and since $\overline{\lin \cone (C-x)}$ is a linear subspace of $X$, we have $\cone \overline{\lin \cone (C-x)} = \overline{\lin \cone (C-x)}$, using \eqref{eq:technical045} we obtain
\[
\overline{\cone (C - x)}\subseteq \overline {\cone \overline {\lin \cone (C-x)}} = \overline {\lin \cone (C-x)}.
\]
Since  $\lin \cone (C-x)\subseteq \cone (C-x)$, this yields 
\[
\overline {\lin \cone (C-x)} = \overline{\cone (C - x)},
\]
hence, $\overline{\cone (C - x)}$ is a linear subspace, and by the definition of quasi-relative interior, we have $x\in \qri C$. We conclude that $\fri C \subseteq \qri C$. 
\end{proof}

\begin{remark}If $X$ is finite-dimensional, or $\ri C\neq \emptyset$, we have $\ri C = \qri C$, and hence all four notions coincide (see \cite[Theorem~2.12]{BorweinGoebel}). Note however that $\icr C \neq \emptyset$ doesn't guarantee that $\icr C = \fri C$ (and hence that $\icr C = \qri C$) ( see Example~\ref{eg:icrneqfri}). 
\end{remark}

A consequence of Theorem \ref{thm:interiorsandwich} is the following result about minimal faces. 

\begin{corollary}\label{lem:eachpointinqri} Let $C$ be a convex subset of a topological vector space $X$. Then for any $x\in C$ we have $x\in \icr F_{\min}(x,C)\subseteq \fri F_{\min}(x,C)\subseteq \qri F_{\min}(x,C)$.
\end{corollary}
\proof
Since $F_{\min}(x,C)=F_{\min}\left(x,F_{\min}(x,C)\right)$, and by \cite[Corollary~3.14]{mill-rosh} the intrinsic core of a minimal face contains $x$ and is hence nonempty, we have 
\[
x\in\icr F_{\min}(x,C)\subseteq \fri F_{\min}(x,C)\subseteq \qri F_{\min}(x,C).
\]
\endproof

Our next goal is to show that the face relative interior is a new notion that doesn't coincide with neither the intrinsic core nor the quasi-relative interior. We consider three examples: the first one is Example~\ref{eg:icrneqfri} that gives a simple construction of a convex set in $l_2$ for which all notions are different, while the intrinsic core is nonempty. The remaining three examples are well-known from the works \cite{BorweinGoebel} and \cite{BorweinLewisPartiallyFinite}. In Example~\ref{eg:02a} the intrinsic core is empty, but the face relative interior isn't; in Example~\ref{exa1} the face relative interior coincides with the intrinsic core, but is a proper subset of the quasi-relative interior. 

\begin{example}\label{eg:icrneqfri} Let $L=l_1$, and embed $L$ into $l_2$ as a subset. Take any point $u\in l_2\setminus l_1$ and let $C= \co (L\cup \{u\})$. 

Every point $x\in C$ can be uniquely represented as $x= (1-\alpha) l + \alpha u$, with $l\in L$. If $\alpha = 1$, then $x=u$, and $F_{\min}(x,C) =F_{\min} (u,C) = \{u\}$. If $\alpha = 0$, then $x\in L$, and $F_{\min}(x,C) = L$. 
Otherwise, for $\alpha \in (0,1)$, for any $y= (1-\beta)l + \beta u\in C$ let $0<\gamma < \min\left\{1,\frac{\alpha}{|\alpha - \beta|}\right\}$, then for a sufficiently small positive value of $\gamma$ 
\[
z: = \left(1-\frac{\alpha - \gamma \beta}{1-\gamma}\right) l + \frac{\alpha - \gamma \beta} {1-\gamma} u \in C,
\]
and hence we have 
\[
x = (1-\gamma) z + \gamma y,
\]
meaning that $y\in F_{\min}(x,C)$ for every $y\in C$. We conclude that  $\icr C = C\setminus (L \cup \{u\})$.

Since for every point $x\in L$ the minimal face is $L$ itself, and the closure of $L = l_1$ in $l_2$ is the entirety of $l_2$, this means $L\subset \fri C$. Notice though that $u\notin \fri C$, hence $\fri C = C\setminus \{u\}$. At the same time,  $\qri C = C$. 

Finally, observe that $\ri C = \emptyset$ because $\aff C = \R u + l_1$ (and so $\aff C \neq l_2$), therefore all inclusions of \eqref{eq:mainsandwich} are strict (we thank our referee for this suggestion).
    
\end{example}

In the next examples, our calculations are done directly from definitions, which may help further familiarise the reader with our ideas. We will need a couple of technical results to proceed: the first one (Proposition~\ref{prop:minfaceunion}) is a characterisation of a minimal face that was proved in \cite{mill-rosh}, and the second one is a generalisation of the delicate construction of slowly converging sequences from \cite{BorweinGoebel}.

\begin{proposition}\label{prop:minfaceunion} The minimal face $F_{\min}(x,C)$ for $x\in C$, where $C$ is convex, can be represented as
\begin{equation}\label{minface:union}
F_{\min}(x,C) = \bigcup \{ [y,z]\subseteq C, \, x\in (y,z)\}.
\end{equation}
\end{proposition}

The next statement is evident from Proposition~\ref{prop:minfaceunion}.

\begin{corollary}\label{cor:minfsubset}
If $A$ and $B$ are convex subsets of a vector space $X$, $A\subseteq B$ and $x\in A$, then
\[
F_{\min}(x,A)\subseteq F_{\min}(x,B).
\]
\end{corollary}

\begin{proposition}\label{prop:majorant} Let $X = l^d$ with $d\in [1,+\infty)$. Then for any $x\in X$ and any $\delta>0$ there exists $y\in l^d$ such that $\|y\|_d =  \delta$  and for all $\varepsilon>0$ there is an $i \in \N$ with  
\[
|x_i|<\varepsilon y_i.
\]
\end{proposition}
\begin{proof} Fix an $x\in X = l^d$ and $\delta>0$. Let $z\in X$ be such that $z_k>0$ for every $k\in \N$ and $\|z\|_d = \delta$. We define a function $N(k)$ recursively as follows. Since $|x_n|\to 0$, for any $k\in \N$ we can find $N'(k)$ such that
\[
|x_n|<\frac{z_k}{k}\quad \forall n \geq N'(k).
\]
Let $N(1) = N'(1)$, and define the remaining values of the function $N(k)$ recursively, as 
\[
N(k) = \max\{N'(k), N(k-1)+1\}.
\]

Observe that $N(k)$ is injective by construction. We define the sequence $y\in X$ such that for each $n\in \N$
\[
y_n: = \begin{cases}
0,& n\notin N(\N),\\
z_k, & N(k) = n. 
\end{cases}
\]
Note that $\|y\|_d = \|z\|_d = \delta$. 

Now choose any $\varepsilon>0$. There exists $k\in \N$ such that $\varepsilon >1/k$. Since 
\[
|x_n|< \frac{z_k}{k}\quad \forall n \geq N(k)\geq N'(k),
\]
we conclude that for $i=N(k)$
\[
|x_i|< \frac{z_k}{k}\leq \varepsilon y_{i}.
\]
\end{proof}

The next example shows that the intrinsic core is different to the face relative interior. More precisely, we demonstrate that it is possible to have $\emptyset = \icr C \subsetneq \fri C$.

\begin{example}[{From \cite[Example 2.4]{BorweinGoebel}}]\label{eg:02a} Let $C$ be the following subset of the Hilbert space $l_2$,
$$
C = \{x\in l_2\,|\, \|x\|_2\leq 1, x_i\geq 0\; \forall\, i \in \N\}.
$$
We will demonstrate that 
\begin{equation}\label{eq:egicrnotfri}
\emptyset = \icr C \neq \fri C,
\end{equation}
moreover,
\begin{equation}\label{eq:friexample1}
\fri C = \{x\in C\, |\, \|x\|_2<1, x_i >0 \, \forall i \in \N\}.
\end{equation}

Our first step is to show that $\icr C =\emptyset$, which was proved in \cite[Example 2.4]{BorweinGoebel}, but we would like to provide a different proof based on \eqref{eq:icrminface}. We will prove this by showing that for any $x\in C$  there is some $y\in C\setminus F_{\min}(x,C)$, and hence from \eqref{eq:icrminface} it follows that $\icr C = \emptyset$. Take any $x\in C$. By Proposition~\ref{prop:majorant} there exists an $y$ such that $\|y\|_2 = 1$ and for any $\varepsilon>0$ there is an $i\in \N$ with $ \varepsilon y_i >|x_i|$. It is clear that $y\in C$. If $y\in F_{\min}(x,C)$, then from Proposition~\ref{prop:minfaceunion} we know that there exists $z\in C$ such that $x\in (y,z)$. Explicitly we have $x = \alpha y + (1-\alpha) z$, $\alpha \in (0,1)$, and so
\[
z = \frac{1}{1-\alpha} (x - \alpha y).
\]
By our construction of $y$, for $\varepsilon=\alpha$ there is some $i\in \N$ such that $x-\alpha y_i<0$, and hence  
\[
z_i = \frac{1}{1-\alpha}(x_i - \alpha y_i) <0,
\]
which contradicts $z\in C$, and hence $y\notin F_{\min}(x,C)$, and we have hence shown that $\icr C = \emptyset$. 

Next, we are going to calculate the minimal faces of all $x\in C$ explicitly. We already know that for any $x\in C$ the set $F_{\min}(x,C)$ does not contain any sequences that converge faster than $x$ (in an appropriate mathematical sense, explicated in Proposition~\ref{prop:majorant}). Now any other $y\in C$ is such that there exists some $\varepsilon>0$ with $\varepsilon y_i\leq x_i$ for all $i\in \N$. We let 
\[
z_\alpha = x+ \alpha(x-y) = \alpha x + x-\alpha y.
\]
and so for $\alpha \leq \varepsilon$
\[
z_i = \alpha x_i + x_i - \alpha y_i \geq \alpha x_i \geq 0.
\]
If $\|x\|_2<1$, we can choose $\alpha\leq \varepsilon$ such that 
\[
\|z_\alpha\|_2\leq \|x\|_2 + \alpha \|x-y\|_2 \leq 1,
\]
and hence $x\in (y,z_\alpha)$ with $y,z_\alpha \in C$. 

We conclude that
\begin{equation}\label{eq:minface001}
F_{\min}(x,C) = \{y\in C\,|\, \exists \varepsilon >0, |x_i|\geq \varepsilon |y_i| \; \forall \, i \in \N\} \quad \forall x\in C, \; \|x\|_2 <1.
\end{equation}

It remains to consider the case $\|x\|_2=1$. We will show that in this case, the minimal face of $x$ is the singleton $\{x\}$. Indeed, assume on the contrary that $x\in (x+u,x-u)$, where $x-u,x+u\in C$ and $u\neq 0$. Then using the parallelogram law and the fact that $\|x\|=1$, we have 
\[
2 = 2 \|x\|_2^2 \leq 2 \|x\|_2^2+2 \|u\|_2^2 = \|x+u\|_2^2 + \|x-u\|_2^2 \leq 2,  
\]
and we conclude that $u = 0$. Hence
\begin{equation}\label{eq:minface002}
F_{\min}(x,C) = \{x\} \quad \forall x\in C, \|x\|_2=1.
\end{equation}

It follows from \eqref{eq:minface001} that if for $x\in C$ we  have $x_i =0$ for some $i\in \N$, then 
\[
F_{\min}(x,C) \subseteq \{u\, |\, u_i = 0\},
\]
and hence $\overline{F_{\min}(x,C)} \neq C$. Likewise, when $\|x\|_2=1$, we have $\overline{F_{\min}(x,C)} =\{x\} \neq C$. On the other hand, if $x\in C$ is such that $\|x\|_2<1$ and $x_i >0$ for all $i\in \N$, representing any $y\in C$ as the limit of eventually zero sequences $y^k$ with the first $k$ entries coinciding with the first $k$ entries of $y$, we have $\bar y^k \in F_{\min}(x,C)$ and since $\bar y^k \to y$, we conclude that $y\in \overline{F_{\min}(x,C)}$. By the arbitrariness of $y$, we conclude that $x\in \fri C$, and therefore we have \eqref{eq:friexample1}. The rest of the relation \eqref{eq:egicrnotfri} also follows from this representation. 
\end{example}

The next example demonstrates that the face relative interior is different to the quasi-relative interior.

\begin{example}{\cite[Example~2.20]{BorweinLewisPartiallyFinite}}\label{exa1} Let
$$
C:= \{x\in l_2\,|\, \|x\|_1 \leq 1\}.
$$
It was shown in \cite{BorweinLewisPartiallyFinite} that 
\begin{equation}\label{eq:qriEx1}
\qri C = C \setminus \{x\in l_2\,|\, \|x\|_1=1 , \; x_n = 0 \; \forall n> \text{some } N\}.
\end{equation}

We will calculate the face relative interior of $C$ explicitly and show that $\fri C\subsetneq \qri C$.

Observe that for every $x\in C$ such that $\|x\|_1<1$ we have $F_{\min}(x,C)= C$. Indeed, take any $y\in C\setminus \{x\}$, and let 
\[
z = x + \alpha (x-y), \quad \text{where } \; \alpha = \frac{1-\|x\|_1}{\|y-x\|_1}.
\]
Then $z\in l_1$,
\[
\|z\|_1 \leq \|x\|_1 + \alpha \|x-y\|_1 = 1,
\]
and hence $z\in C$. We have 
\[
x = \frac{1}{1+\alpha} z+ \frac{\alpha}{1+\alpha} y\in (y,z),
\]
and from Proposition~\ref{prop:minfaceunion} it follows that $y\in F_{\min}(x,C)$. We conclude that $F_{\min}(x,C) = C$, and so 
\[
\{x\,|\, \|x\|_1<1\} \subseteq \icr C  \subseteq \fri C. 
\]

Now consider the case $\|x\|_1 = 1$. Suppose that $y\in F_{\min}(x,C)$. Then there is some $z\in C$ such that 
\[
x\in (y,z).
\]
Without loss of generality, we can assume that $x_i\geq 0$ for all $i$ (otherwise it is easy to see that the argument still holds by using the relation $x_i = -|x_i|$ instead of $x_i = |x_i|$ whenever necessary). We have for some $\alpha \in (0,1)$
\[
1 =  \sum_{i=1}^\infty |x_i|  = \sum_{i=1}^\infty x_i  = \sum_{i=1}^\infty (\alpha y_i+ (1-\alpha) z_i) = \alpha   \sum_{i=1}^\infty y_i+ (1-\alpha) \sum_{i=1}^\infty z_i \leq \alpha \|y\|_1 + (1-\alpha) \|z\|_1\leq 1.
\]
Note that,  $$\sum_{i=1}^{\infty}y_i\leq \sum_{i=1}^{\infty}|y_i|=1,  \sum_{i=1}^{\infty}z_i\leq \sum_{i=1}^{\infty}|z_i|=1,$$ and using  $$\alpha   \sum_{i=1}^\infty y_i+ (1-\alpha) \sum_{i=1}^\infty z_i=1,$$ this implies that $\sum_{i=1}^{\infty}y_i=\sum_{i=1}^{\infty}z_i=1$. Then it follows that 
$$ 
\sum_{i=1}^\infty |y_i|-\sum_{i=1}^\infty y_i=\sum_{i=1}^\infty (|y_i|-y_i)=0,
$$ 
implying that $|y_i|=y_i\geq 0$ for all $i\in \mathbb{N}$. Hence 
\[
F_{\min}(x,C) \subseteq  C\cap \{y\,|\, \sum_{i=1}^\infty y_i = 1, y_i\geq 0, \forall i\in \mathbb{N}\}.
\]
Then, if $x\in (\ell_2)_+$ we get $F_{\min}(x,C) \subseteq \overline{F_{\min}(x,C)}\subset (\ell_2)_+$, this guarantees that $\overline{F_{\min}(x,C)} \nsupseteq C$, since $C$ contains elements with strictly negative entries.

The remaining cases when some of the entries of $x$  are negative can be considered similarly, by taking care of the signs of $y_i$ and $z_i$. We conclude that $\fri C = \{x\,|\, \|x\|_1<1\}$, which is strictly smaller than the quasi-relative interior. It is worth noting that in this case, the notion of face relative interior appears to better reflect the intuitive perception of what the `relative interior' of the set $C$ should be.

To conclude, in this case we have 
\[
\icr C = \fri C \subsetneq \qri C = \qi C.
\]

\end{example}

Our last example shows that there exists a convex set with nonempty intrinsic core, that is also different to its face relative interior. Before we need to recall that for any subset $S$ of a real vector space $X$, we define its convex hull $\co S$ as the set of all finite \emph{convex combinations} of points from $S$.

\section{The question of nonemptiness}\label{sec:nonemptiness}

While the intrinsic core can be empty (and in fact any infinite-dimensional vector space contains a convex set whose intrinsic core is empty, see \cite{Holmes}), the quasi-relative interior of a convex set $C$ is guaranteed to be nonempty under the assumption that  $C$ is CS-closed, and some additional conditions imposed on the ambient space.

We discuss the notion of CS-closedness in Section~\ref{ss:csclosed} and give a comprehensive list of conditions under which convex sets in general vector spaces are CS-closed. This list is similar to the one given in Proposition~6 in \cite{BorweinConvexRelations}. We then move on to the nonemptiness results in Section~\ref{ss:nonempty}, and prove in Theorem~\ref{thm:frinonempty} that the face relative interior of a CS-closed set is nonempty under the assumptions used in \cite{BorweinLewisPartiallyFinite} to prove that the quasi-relative interior is nonempty. In particular, it follows that the face relative interior of any convex set is nonempty in a separable Banach space.

\subsection{CS-closed sets}\label{ss:csclosed}

It may happen that the convex hull of a compact convex subset of a topological vector space $X$ is not closed, as shown in the following example.

\begin{example}[{From \cite[Example 5.34]{hitchhiker}}]\label{eg:hullnotclosed} Let $(u^n)_{n\in \N}$ be a sequence of points in  $l_2$ such that each $u^n$ is a sequence of zeros except for the $n$-th coordinate which is $1/n$. The set $A = \{0_{l_2}, u^1,u^2,\dots, u^n,\dots \}$ is compact, but its convex hull $\co A$ is not. Indeed, consider the sequence $(x^k)_{k\in \N}$, where 
$$
x^k = \sum_{i=1}^k \frac{1}{2^i} u^i  +\frac{1}{2^k}0_{l_2}.
$$  
We have $x^k \in \co A$ by the definition of the convex hull, moreover,
$$
x^k \to x = \sum_{i=1}^{\infty}\frac{1}{2^i} u^i.
$$
It is evident that $x\in l_2$, and hence also $x\in \overline{\co} A$, but $x\notin \co A$. 
\end{example}

The point $x$ in Example~\ref{eg:hullnotclosed} can be viewed as an infinite generalisation of a convex combination of points in $A$. Indeed, given an arbitrary set $S\subseteq X$, we can define the infinite convex hull (or $\sigma$-convex hull, see \cite{SigmaConvex}) as
\begin{equation}\label{eq:cs}
\tilde \co S = \left\{ \sum_{i\in \N}\lambda_i x_i \;\Bigr|\; \sum_{i\in \N}\lambda_i =1,\quad \lambda_i\geq 0,\,  x_i\in S \;\,  \forall \, i \in \N  \right\},
\end{equation}
where only convergent series $\sum_{i\in \N}\lambda_i x_i$ are taken into account.

A subset $S$ of a normed vector space $X$ is called convergent-series closed, or CS-closed for short (see \cite{Jameson}) if $S=\tilde\co S$, i.e. $S$ contains every convergent sum $\sum_{i\in \N}\lambda_i x_i$ as in \eqref{eq:cs}, and CS-compact if in addition every such series converges. Every CS-compact set is CS-closed, and every CS-closed set is convex, see \cite{FremlinTalagrand}. Even though every closed convex set in a normed vector space $X$ is CS-closed, CS-closed sets do not need to be closed: for instance, the open unit ball in any normed space is CS-closed. In fact, \emph{all} convex sets in completely metrizable Fr\'echet spaces are CS-closed due to \cite[Corollary~5]{FremlinTalagrand}. In particular, \emph{all convex subsets of a Banach space are CS-closed}. 

\begin{proposition}\label{prop:CS} Let $X$ be a Hausdorff topological vector space and let $C$ be a convex subset of $X$. The following conditions are sufficient for $C$ to be CS-closed. 
\begin{enumerate}
    \item \label{it:openisCS} The set $C$ is open.
    \item \label{it:closedisCS} The set $C$ is sequentially closed.
    \item \label{it:intersection} The set $C$ is an intersection of CS-closed sets.
    \item \label{it:separatedfinite} The set $C$ is finite-dimensional.
    \item \label{it:Gdelta} The space $X$ is Fr\'echet (locally convex and metrizable) and the set $C$ is $G_\delta$ (a countable intersection of open sets).  
    \item \label{it:allCS} If $X$ is completely metrizable (that is, the topology of $X$ is induced by a metric under which $X$ is complete).
\end{enumerate}
\end{proposition}
\begin{proof}
    For item \ref{it:openisCS} see Example (i) in \cite{JamesonConvexSeries}, item~\ref{it:closedisCS} is shown in \cite[Proposition 1]{JamesonConvexSeries}. Item \ref{it:intersection} is obvious and is also mentioned in \cite{JamesonConvexSeries}. Items~\ref{it:Gdelta}~and~\ref{it:allCS} are proved in \cite[Corollaries~4~and~5]{FremlinTalagrand}.
\end{proof}

\subsection{Nonemptiness of the face relative interior}\label{ss:nonempty}

The main goal of this section is to show that the face relative interior is nonempty under reasonable assumptions. We show that the same conditions that are known to guarantee the nonemptiness of quasi-relative interior ensure the nonemptiness of the face relative interior. Before recalling the key result from \cite{BorweinLewisPartiallyFinite} given in Theorem~\ref{thm:qrinonempty} below (that we then mirror in Theorem~\ref{thm:frinonempty} for the face relative interiors), we note here that there is an earlier result {\cite[Proposition~1.2.9.]{zalinescubook}} that all nonempty CS-complete sets in first countable separable locally convex spaces have nonempty quasi-relative interiors. Together with item~\ref{it:allCS} of Proposition~\ref{prop:CS} either of these results implies that every nonempty convex subset of a separable Banach space has nonempty quasi-relative interior (see \cite[Theorem 2.8]{BorweinGoebel}). We show that the same is true for the face relative interior (see Corollary~\ref{cor:friBanach}).

\begin{theorem}[{\cite[Theorem 2.19]{BorweinLewisPartiallyFinite}}]\label{thm:qrinonempty} Suppose $X$ is  a topological vector space such that either
\begin{enumerate}
\item  $X$ is a separable Fr\'echet space, or
\item  $X = Y^*$ with $Y$ a separable normed space, and the topology on $X$ given by the weak* topology of $Y^*$.
\end{enumerate}
Suppose that $C\subseteq X$ is CS-closed. Then $\qri C\neq \emptyset$.
\end{theorem}

We will need the following technical result from \cite{BorweinLewisPartiallyFinite}.

\begin{lemma}[{\cite[Lemma~2.18]{BorweinLewisPartiallyFinite}}]\label{lem:218}    Suppose that $X$ is a topological vector space with either
\begin{enumerate}
    \item $X$ separable and complete metrizable, or
    \item $X=Y^*$ with $Y$ a separable normed space, and the topology on $X$ given by the weak* topology of $Y^*$.
\end{enumerate}
If $C\subseteq X$ is nonempty and CS-closed, then there exist sequences $(x_n)_{n=1}^\infty$ and $(\lambda_n)_{n=1}^\infty$, such that $\sum_{i=1}^\infty \lambda_i = 1$, $\lambda_i > 0 $ and $x_i \in C$ for all $i \in \{1,\dots\}$, the sequence $(x_n)_{n=1}^\infty$ is dense in $C$, and the convex series $\sum_{i=1}^\infty \lambda_i x_i$ converges to some $\bar x \in C$. 
\end{lemma}

\begin{theorem}\label{thm:frinonempty}
 Suppose that $X$ is a topological vector space with either
\begin{enumerate}
    \item $X$ separable and complete metrizable, or
    \item $X=Y^*$ with $Y$ a separable normed space, and the topology on $X$ given by the weak* topology of $Y^*$.
\end{enumerate} If $C$ is a nonempty CS-closed subset of $X$, then $\fri C\neq \emptyset$.
\end{theorem}
\proof By Lemma~\ref{lem:218} there exists a sequence $(\lambda_n)_{n\in \N}$ of positive real numbers with $\sum_{n\in \N}\lambda_n=1$ and a sequence $(x_n)_{n=1}^\infty$ dense in $C$ such that $\bar{x}:=\sum_{n\in \N}\lambda_n x_n\in C$. 

We will show that $C\subseteq\overline{F_{\min}(\bar{x},C)}$, hence $\bar x \in \fri C \neq \emptyset$.  Indeed, for all $n\in \N$ we have
 $$\bar{x}=\lambda_n x_n+(1-\lambda_n)\sum_{k\in \N,k\neq n}\frac{\lambda_k}{1-\lambda_n}x_k,$$ 
 using the SC-closedness of $C$, we have that $\sum_{k\in \N,k\neq n}\frac{\lambda_k}{1-\lambda_n}x_k\in C$, then for all $n\in \N$, $x_n\in F_{\min}(\bar{x},C)$. Taking in account that $(x_n)_{n\in\N}$ is dense on $C$, we have $C\subseteq \overline{F_{\min}(\bar{x},C)}$, implying that $\bar{x}\in \fri C$.  
 
\endproof

\begin{corollary}\label{cor:friBanach} Any nonempty convex subset of a separable Banach space has nonempty face relative interior. 
\end{corollary}
\begin{proof} It follows from Theorem~\ref{thm:frinonempty} that any CS-closed subset of a separable Banach space has nonempty face relative interior, while item~\ref{it:allCS} of Proposition~\ref{prop:CS} ensures that every convex subset of a separable Banach space is CS-closed.
\end{proof}

Since the face relative interior is guaranteed to be nonempty under the same (known) conditions as the quasi-relative interior, and we couldn't find any examples of convex sets for which the face relative interior is empty, but the quasi-relative interior isn't, we conjecture that the nonemptiness of $qri$ must guarantee the nonemptiness of $fri$.

\begin{conjecture} Let $C$ be a convex subset of a topological vector space $X$. The face relative interior of $C$ is nonempty whenever the quasi-relative interior of $C$ is nonempty. 
\end{conjecture}

\section{Calculus of face relative interiors}\label{sec:calculus}

The purpose of this section is to lay out some calculus rules and other properties of the face relative interior, and compare these results to what is known as the quasi-relative interior. We draw explicit parallels between these notions and their properties, as we work through the relevant statements and their proofs. 

We begin with the most straightforward properties.

\begin{proposition}[Basic properties] \label{prop:basic} If $C$ is a convex subset of a topological vector space $X$, then the following properties hold.
\begin{enumerate}
    \item \label{it:convex} The set $\fri C$ is a convex subset of $C$.
    \item \label{it:segmentfri} Given $x\in \fri C$ and any $y\in C$, we have $[x,y)\subseteq \fri C$.    
    \item \label{it:closure} If $\fri C\neq \emptyset$, then  $\overline{C}= \overline{\fri C}$.
    \item \label{it:idempotent} The face relative interior operation is idempotent, that is, $\fri(\fri C)=\fri C$.
\end{enumerate}
\end{proposition}
\begin{proof}
The first assertion \ref{it:convex} is easy to prove, and we have already addressed it in Proposition~\ref{prop:friconvex}. Note also that \ref{it:convex} follows from \ref{it:segmentfri}.

To demonstrate that \ref{it:segmentfri} is true, let $x\in \fri C$ and $y\in C$. By definition of a face for any $u\in (x,y)$ we have $x\in F_{\min}(u, C)$, therefore 
\[
C\subseteq \overline{F_{\min}(x,C)} \subseteq \overline{F_{\min}(u,C)},
\]
hence $u\in \fri C$. We conclude that $[x,y)\subseteq \fri C$.

To see that \ref{it:closure} is true, assume that $\fri C\neq \emptyset$. Then there exists some $x\in \fri C$. Now for any $y\in C$ by \ref{it:segmentfri} we have $(x,y)\subset \fri C$, and hence $y\in \overline {\fri C}$. Therefore $C\subseteq \overline {\fri C}$, from where \ref{it:closure} follows. 

Finally, we focus on proving \ref{it:idempotent}. By \ref{it:convex} the set $\fri C$ is convex, therefore the composition $\fri(\fri C)$ is well-defined. By the definition of $\fri$, we have $\fri(\fri C)\subseteq \fri C$. It remains to show that $\fri C\subseteq \fri(\fri C)$. For this it is sufficient to demonstrate that $F_{\min}(x,C)\subseteq \overline{F_{\min}(x,\fri C)}$ for any $x\in \fri C$ (indeed, if for any $x\in \fri C$ we have $F_{\min}(x,C)\subseteq\overline{F_{\min}(x,\fri C)}$ then $\fri C\subseteq C\subseteq \overline{F_{\min}(x,C)}\subseteq \overline{F_{\min}(x,\fri C)}$, and so $x\in \fri (\fri C)$). 

Let $y\in F_{\min}(x,C)$. From Proposition \ref{prop:minfaceunion} we have $x\in (y,z)$, where $z$ is also in $F_{\min}(x,C)$. By \ref{it:segmentfri} we conclude that $(y,z)\subset \fri C$. Then by the definition of a face we have $(y,z)\subseteq F_{\min}(x,\fri C)$, and hence $y\in \overline{F_{\min}(x,\fri C)}$, proving the assertion. 
\end{proof}

\begin{remark} Note that the quasi-relative interior satisfies all properties given in Proposition~\ref{prop:basic}. The relevant results can be found in Proposition~2.11 (for \ref{it:convex}), Lemma~2.9 (for \ref{it:segmentfri}),  Proposition~2.12 (for \ref{it:closure}) of \cite{BorweinLewisPartiallyFinite} and Proposition 2.5 (vii) (for \ref{it:idempotent}) of \cite{Regularity}. In the case of the quasi-relative interior, these properties are proved under the assumption that the space $X$ is locally convex. This condition is needed to carry out the proofs that rely on dual objects, which is not required in our case.
\end{remark}

\begin{proposition}[More calculus rules]\label{prop:prodandlin} Let $X$ and $Y$ be topological vector spaces. 
\begin{enumerate}
    \item \label{it:prod} For any pair of convex sets $C \subseteq X$ and $D\subseteq Y$ we have 
    \[
    \fri (C\times D) = \fri C \times \fri D.
    \]
    \item \label{it:linimage} For any continuous linear operator $T:X\to Y$ and any convex set $C\subseteq X$ we have 
    \begin{equation}\label{eq:linear}
    T(\fri C) \subseteq \fri (T(C)).
    \end{equation}
    
\end{enumerate}
   
\end{proposition}
\begin{proof} The equality \ref{it:prod} is easy to see using the observation that $F_{\min}((c,d),C\times D) = F_{\min} (c,C)\times F_{\min}(d,D)$ for any $(c,d)\in C\times D$.

To show \ref{it:linimage}, notice that since the linear mapping $T$ is continuous, the preimages of closed sets must be closed, hence for any subset $S$ of $X$ we must have $T(\overline{S}) \subseteq \overline{T S}$. To complete the proof it remains to show that for any $x\in C$ we have $T(F_{\min}(x, C))\subseteq F_{\min}(T(x), T(C))$, in fact, for any $x\in \fri C$ we have $C\subseteq \overline{F_{\min}(x,C)}$, and hence 
\[
T(C) \subseteq T(\overline{F_{\min}(x,C)}) \subseteq \overline{T(F_{\min}(x,C)})\subseteq \overline{F_{\min}(T(x), T(C))},
\]
therefore $T(x) \in \fri T(C)$.

Now take any $x\in C$, and suppose that $u\in F_{\min}(x,C)$. By Proposition~\ref{prop:minfaceunion} there must exists some $v\in F_{\min}(x,C)$ such that $x\in (u,v)$. Since $T$ is a linear mapping, this means that 
$T(x) \in (T(u),T(v))$, therefore $T(u)\in F_{\min}(T(x), T(C))$.  This concludes the proof.

%Assume that $T$ is injective on $C$, by consequence we have $T:C\rightarrow T(C)$ as a bijection. Let $\hat{x}\in C$ such that $T(\hat{x})\in \fri T(C)$, then $T(C)\subseteq \overline{F_{\min}\left(T(\hat{x}),T(C)\right)}$. Now for all $x\in C$ and for each neighbourhood $U_x$ of $x$, consider the set $V_{Tx}:=T(U_x)$, then $T(z)\in V_{Tx}$ for all $z\in U_x$. Take $z\in U_x\cap C$ with  $T(z)\in V_{Tx}\cap F_{\min}(T(\hat{x}),T(C))$. Such $z$ exists because $F_{\min}(T(\hat{x}),T(C))$ is dense in $T(C)$, and $F_{\min}(T(\hat{x}),T(C))\subseteq T(C)$. % In fact, if for all $z\in  U_x\cap C$, $T(z)\notin F_{\min}(T(\hat{x}),T(C))$ means that $T(x)\notin \overline{F_{\min}(T(\hat{x}),T(C))}$, because $T(x)\in T(U_x)\cap T(C)$ and $T(U_x)$ is a neighbourhood of $T(x)$ in $T(C)$. 
%Since $T(z)\in F_{\min}\left(T(\hat{x}),T(C)\right)$, using Proposition~\ref{prop:minfaceunion}, there exist $y_z\in C$ with $T(y_z)\in F_{\min}\left(T(\hat{x}),T(C)\right)$ and $\alpha_z\in (0,1)$, such that $T(\hat{x})=\alpha_{z}T(z)+(1-\alpha_z)T(y_z)=T(\alpha_z z+(1-\alpha_z)y_z)$, then by the injectivity of $T$ and convexity of $C$ we have  $\hat{x}=\alpha_z z+(1-\alpha_z)y_z$, this implies that $z\in F_{\min}(\hat{x},C)$, then $x\in \overline{F_{\min}(\hat{x},C)}$, the arbitrarily of $x$ implies that $C\subseteq \overline{F_{\min}(\hat{x},C)}$, then $\hat{x}\in\fri C$ and $T(\hat{x})\in T(\fri C)$.

\end{proof}

\begin{remark} The product rule in Proposition~\ref{prop:prodandlin} \ref{it:prod} also holds for the quasi-relative interior (see \cite[Proposition~2.5]{BorweinLewisPartiallyFinite}). The counterpart of inclusion \ref{it:linimage} for the quasi-relative interior is given in Proposition 2.21 of \cite{BorweinLewisPartiallyFinite}.
\end{remark}

Our next goal is to address a crucial calculus question: Under what conditions does the sum of face relative interiors equal the face relative interior of the sum? It is not difficult to see that translations preserve face relative interiors, as we show next. 

\begin{remark}\label{prop:fritranslation}
Let $C$ be a convex set in a topological vector space $X$, and suppose that $a\in X$, then $\fri (C+a)=\fri C +a$. 
\end{remark}
\proof
Let $x\in  C+a$ then, $x=x_c+a$ with $x_c\in C$. Note that
\begin{equation}\label{minfa}
\overline{F_{\min}(x_c,C)+a}=\overline{F_{\min}(x_c,C)}+a 
\end{equation}. For any face $F$ of $C$ we have that $F\lhd C \Longleftrightarrow F+a\lhd C+a$, which implies that $F_{\min}(x_c+a,C+a)=F_{\min}(x_c,C)+a$. Using the last fact in \eqref{minfa} we have,
$$\overline{F_{\min}(x_c+a,C+a)}=\overline{F_{\min}(x_c,C}+a)=\overline{F_{\min}(x_c,C)}+a.$$
The last equation implies that $C\subseteq \overline{F_{\min}(x_c+a,C+a)}$ is equivalent to $C\subseteq \overline{F_{\min}(x_c,C)}+a$. Then $\fri(C+a)=\fri C +a$. 
\endproof

\begin{proposition}\label{sum}
Let $C$ and $D$ be convex subsets of a topological vector space $X$. Then   
\[
\fri C+\fri D\subseteq \fri(C+D).
\]
\end{proposition}
\begin{proof} Let $C$ and $D$ be convex subsets of a topological vector space $X$, and let $T:X\times X\to X$ be the addition mapping, $T(x,y) = x+y$. This mapping is linear and continuous.

Using Proposition~\ref{prop:prodandlin} we obtain
\begin{align*}
\fri C + \fri C & = T(\fri C \times \fri D) \quad \text{by the definition of $T$}\\
                & = T(\fri (C\times D)) \quad \text{by Proposition~\ref{prop:prodandlin}~\ref{it:prod}}\\    
                & \subseteq \fri T(C\times D) \quad \text{by Proposition~\ref{prop:prodandlin}~\ref{it:linimage}}\\    
                & =  \fri (C+D)\quad \text{by the definition of $T$}.
\end{align*}
\end{proof}

\begin{remark} An analogue of Proposition~\ref{sum} holds in the case of the quasi-relative interior (see \cite[Lemma~3.6~(b)]{BorweinGoebel}). We refer the reader to  \cite{ZalinescuThreePb} for a nuanced discussion of this property of quasi-relative interiors. We note that the sum of quasi-relative interiors doesn't necessarily coincide with the quasi-relative interior of the sum even when the individual quasi-relative interiors are nonempty, and the  same is true for the face relative interior, as we show in what follows using an example from \cite{ZalinescuThreePb}.
\end{remark}

\begin{example}[Example~2 in  \cite{ZalinescuThreePb}] Let $X=\ell_2$, $\bar{x}=(n^{-1})_{n\geq 1}$, $C=[0,1]\bar{x}$ and $D=\ell_1^+$. 

For any $y\in C+D$ and $n\in \N$ let $y^n$ be the truncated sequence whose first $n$ entries coincide with $y$, and the remaining entries are all zeros. Note that $y^n\in C+D$, since $0\in C$ and $y^n\in l_1^+$. It is also evident that $y^n\to y$. 

For any $x\in l_1^{++}$ there is a suffciently small $\alpha>0$ such that $x + \alpha (x-y^n)$ has strictly positive entries, hence, for any $x\in l_1^{++}$ we have $y^n\in F_{\min}(x,C+D)$. We conclude that $l_1^{++}\subseteq \fri (C+D)$. 

On the other hand, observe that since $\fri C = (0,\bar x)$, we have $l_1^{++}\cap \fri C = \emptyset$, hence $\fri C+\fri D $ doesn't contain elements of $l_1^{++}$. We conclude that $\fri C+\fri D$ is strictly smaller than $\fri (C+D)$. 
\end{example}

The next example shows that the intersection rules that are true for the quasi-relative interiors do not hold for the face relative interior (see \cite[Proposition~3]{ZalinescuThreePb}). Specifically, there are convex sets for which $\fri (C\cap D)\nsubseteq \fri C \cap \fri D$ even under the condition $\fri C \cap \fri D\neq \emptyset$.

\begin{example}\label{eg:intersectionfails} We construct the sets $C$ and $D$ in $l_2$ such that $0\in \fri (C\cap D)$, and at the same time $0\notin  \fri C \cap \fri D\neq \emptyset$. 

Let $v = \left(1,\frac 1 2,\frac 1 3, \dots, \frac 1 n, \dots\right)$ and consider the following subsets of $l_2$:
\begin{align*}
A & = \{ t v + u \,| u\in l_1, t>0\} ,\\
B & = \{ x\in l_1\, | \, x_i \geq -1 \; \forall i \},\\
C & = \{ v+ s ( w-v)\, | \, w\in B , \, s> 0 \},\\
D & = A \cup B.
\end{align*}
First observe that all of these sets are convex, including $D$: both $A$ and $B$ are convex, while for any $x \in A$ and $y\in B$ and any $\alpha \in (0,1)$ we have 
\[
(1-\alpha ) x+ \alpha y = (1-\alpha ) (u+ t v)+ \alpha y =  
(1-\alpha )  t v + (1-\alpha ) u+ \alpha y \in A. 
\]
Another important observation is that $v\in \overline B$ and $v\notin l_1$.

It is also evident by considering the intersections $C\cap A$ and $C\cap B$ that  
\[
C\cap D = \{ v + s (w-v)\, | \, w\in B, \, s\in (0,1]\}.
\]

We will show that $0\in \fri (C\cap D)$, and at the same time $0\notin  \fri C \cap \fri D\neq \emptyset$. 

To show that $\fri C\cap \fri D\neq \emptyset$, we demonstrate that $\frac v 2 \in \fri C$ and $\frac v 2 \in\fri D$.  Take any $x\in C$. Then $x= v+s (w-v)$ for some $w\in B$. Choose $\varepsilon>0$ such that 
\[
\frac 1 2 + \varepsilon \left ( \frac 1 2 -s \right ) > 0, \qquad \varepsilon \left ( s\left ( \|w\|_1+1\right)-\frac 1 2 \right) <\frac 1 2, 
\]
Then 
\[
y = \frac 1 2 v + \varepsilon \left ( \frac 1 2 v -x \right ) = (1-t') v + t' w'\in C,
\] 
where
\[
t' = \frac 1 2 + \varepsilon \left ( \frac 1 2 -s \right ) , \quad w'= \frac{-\varepsilon s \|w\|_1}{t'}\in B.
\]
Therefore $\frac v 2 \in (x, y)$, and hence $x\in F_{\min} (\frac v 2 , C)$. We conclude that $C =  F_{\min} (v/2, C)$, hence $\frac v 2 \in \icr C \subseteq \fri C$ (by Theorem~\ref{thm:interiorsandwich}).

Likewise, for any $x\in D$ we have $x = t v + u$, where $t\geq 0$ and $u\in l_1$. Choose a sufficiently small $\varepsilon$ such that 
\[
\frac 1 2 + \varepsilon \left ( \frac 1 2 -t \right ) > 0,
\]
then 
\[
\frac 1 2 v + \varepsilon \left ( \frac 1 2 v -x \right ) = \left(\frac 1 2 + \varepsilon \left ( \frac 1 2 -t \right ) v - 
\varepsilon u \right)\in A \subset D,
\]
and using the same argument as before, we conclude that $\frac v 2 \in \fri D$.

To show that  $0\in \fri (C\cap D)$, take any $x\in B$, then either $x=0 $ or $-x/\|x\|_1\in B$. We conclude that $B\subseteq F_{\min} (0, C\cap D)$. Since $C\cap D \subseteq \overline B$, we conclude that $0\in \fri C\cap D$. 

To show that $0\notin \fri C\cap \fri D$, it is sufficient to demonstrate that $0\notin \fri D$. It is evident that for any $x\in A$ there is no $y\in D$ such that $0\in (x,y)$ (since $y$ can't have a negative ``$v$ coordinate''). Therefore $ F_{\min} (0, D)\subseteq B$. Now for any $x\in \overline B$ we must have all coordinates of $x$ not smaller than $-1$, and hence, say,  the point $v+(-2,0,0,\dots)$ is in $D$ but is not in $\overline B$. We hence conclude that $D\notin  \overline {F_{\min} (0, D)}$, and $0\notin \fri D$.

\end{example}

Our last calculus result relates the face relative interiors of sets and their closures in a more refined way than \ref{it:closure} of Proposition~\ref{prop:basic}.

\begin{proposition}\label{icrfri}
    Let $A,B$ convex sets in the real topological space $X$, such that $A\subseteq B$, $\icr A=\fri A$ and $A\cap \fri B\neq \emptyset$, then $\fri A\subseteq \fri B$.
\end{proposition}
\begin{proof}
    Consider $x_0\in \fri A=\icr A$, then  from the definition of intrinsic core (see page~\ref{page:deficr} or Definition~3.1 in \cite{mill-rosh}), for all $x\in A$ there exist $y\in A$ with $x_0\in(y,x)$. Take $\bar{x}\in A\cap \fri B$, there exist $\bar{y}\in A\subseteq B$ such that $x_0\in (\bar{x},\bar{y})$. Now, since $\bar{y}\in  B$  and $\bar{x}\in A\subseteq B$, using Proposition \ref{prop:basic}\ref{it:segmentfri} we have that $[\bar{x},\bar{y})\subset\fri B$, which implies that $x_0\in \fri B$, the arbitrarily of $x_0$ implies that $\fri A\subseteq\fri B$. 
\end{proof}
\begin{proposition}[Intersection property]\label{prop:intersection} Let $C$ and $D$ be nonempty convex sets in a real topological space $X$ such that $\icr (C\cap D)=\fri (C\cap D)$. Then
\begin{enumerate} 
    \item If $C\cap \fri D\neq \emptyset$, then $\fri (C\cap D)\subseteq C\cap \fri D$.
    \item If $\fri C\cap \fri D\neq \emptyset$, then $\fri (C\cap D)\subseteq \fri C\cap \fri D$.
\end{enumerate}
    
\end{proposition}
\begin{proof}
\begin{enumerate}
    \item This is a direct consequence of Proposition \ref{icrfri} with $A=C\cap D$ and $B=D$.
    \item Since $\fri C\cap\fri D\neq \emptyset$, then  $C\cap\fri D\neq \emptyset$ and $D\cap\fri C\neq\emptyset$, using twice the previous item, we have that $\fri(C\cap D)\subseteq C\cap\fri D$ and $\fri(C\cap D)\subseteq D\cap\fri C$, which implies that $$\fri(C\cap D)\subseteq \big(C\cap\fri D\big)\cap \big(D\cap\fri C\big)=\fri C\cap\fri D.$$
\end{enumerate}
\end{proof}

We finish this section with a characterisation of face relative interior that we then use to prove some relations between face relative interiors of convex sets and their closures in Proposition~\ref{prop:closures}.

\begin{theorem}\label{notinfri} Let $X$ be a topological vector space. Given a convex set $C\subset X$ and $x_0\in C$, we have $x_0\notin \fri C$ if and only if there exist $y\in C$ and a neighbourhood $\mathrel{V_y}$ of $y$ such that for all $z\in \mathrel{V_y} \cap C$, the direction $x_0 -z$ is a not feasible direction, that is, for all $\epsilon > 0$, $x_0+\epsilon(x_0-z)\notin C$.
\end{theorem}
We provide a geometric illustration of Theorem~\ref{notinfri} and its proof in Fig.~\ref{fig:technical}.
\begin{figure}[ht]
    \centering
    \includegraphics[width=0.5\textwidth]{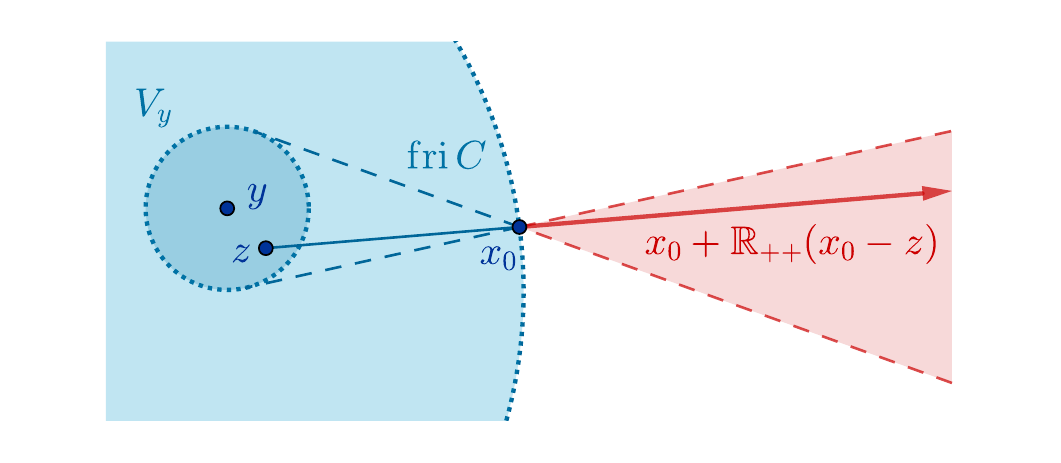}
    \caption{A technical construction in the proof of Theorem~\ref{notinfri}.}
    \label{fig:technical}
\end{figure}
\begin{proof}[Proof of Theorem~\ref{notinfri}]
Suppose that $x_0\in C\setminus  \fri C$. Then $C\nsubseteq \overline{F_{\min}(x_0,C)}$ and there exists $y\in C$ such that $y\notin \overline{F_{\min}(x_0,C)}$, which implies the existence of a neighborhood $\mathrel{V_y}$ such that, $\mathrel{V_y}\cap F_{\min}(x_0,C)=\emptyset$. For any $z\in \mathrel{V_y} \cap C$ and $\epsilon>0$ we have $x_0\in (z,x_0+\epsilon (x_0-z))$, which together with $z\notin F_{\min}(x_0,C)$ and \eqref{minface:union} yields that  for all $\epsilon>0$ the point $x_0+\epsilon(x_0-z)\notin C$. 

Conversely, for some $x_0\in C$ let $y\in C$ and $\mathrel{V_y}$ be such that for all $\epsilon>0$ and $z\in \mathrel{V_y}\cap C$, $x_0+\epsilon(x_0-z)\notin C$. Then for all $z\in \mathrel{V_y}\cap C$ (as well as for $z\notin C$) we have $z\notin F_{\min}(x_0,C)$, and therefore $y\notin \overline{F_{\min}(x_0,C)}$.
\end{proof}

\begin{proposition}\label{prop:closures}
Given convex sets $A, B$ such that $A\subseteq B\subseteq \overline{A}$. Then we have:
\begin{enumerate}
\item \label{it:a} $\fri A \subseteq \fri B\subseteq \fri \overline{A}$.
\item \label{it:b} If $A$ is an open or a closed set, then $\fri A=A\cap \fri B=A\cap \fri \overline{A}$.
\end{enumerate}
\end{proposition}
\begin{proof}
 Note that given a point $x\in A$, since $A\subseteq B \subseteq \overline{A}$ we have by Corollary~\ref{cor:minfsubset} that $F_{\min}(x,A) \subseteq F_{\min}(x,B)\subseteq F_{\min}(x,\overline{A})$.

To show \ref{it:a} take $x\in \fri A$ then, $A\subseteq \overline{F_{\min}(x,A)}$, now since $\overline{F_{\min}(x,A)}$ is closed, we have that $A\subseteq B \subseteq \overline{A}\subseteq  \overline{F_{\min}(x,A)}\subseteq  \overline{F_{\min}(x,B)} \subseteq \overline{F_{\min}(x,\overline{A})}$, proving that $x\in \fri B \subseteq \fri \overline{A}$.

To demonstrate \ref{it:b}, notice that due to \ref{it:a} we have that $\fri A \subseteq A\cap \fri B \subseteq A\cap \fri \overline{A}$. We will prove that $A\cap \fri \overline{A}\subseteq \fri A$ which completes the proof. If $A$ is closed, we have $A=\overline{A}$ and the proof is trivial. Suppose that $A$ is an open set. Take $x\in A\cap \fri \overline{A}$, suppose that $x\notin \fri A$. Using Theorem \ref{notinfri}, we have the existence of $y\in A$ and an open neighbourhood $V_y$ of $y$, such that for all $z\in A\cap V_y$, we have $x+\epsilon (x-z)\notin A$ for all $\varepsilon>0$. Define the continuous function $F_{x,\varepsilon}:X\rightarrow X$ as $F_{x,\varepsilon}(z):=x+\varepsilon(x-z)$, notice that the set $U^y_{x,\varepsilon}:=F_{x,\varepsilon}(V_y\cap A)$ is an open set because $F_{x,\varepsilon}$ is an homeomorphism for fixed $\varepsilon$ and $x$, and that $A\cap V_y$ is open, also $U^y_{x,\varepsilon}\cap A=\emptyset$ and $x+\varepsilon(x-y)\in U^y_{x,\varepsilon}$. This implies that $x+\epsilon(x-y)\notin \overline{A}$ which gets that $x\notin \fri \overline{A}$, contradicting our assumption. We conclude that $\fri A=A\cap \fri \overline{A}$.
\end{proof}

\section{Equivalence classes of faces}\label{sec:equiv}

Recall that any convex set in a real vector space has a beautiful and useful representation as the disjoint union of intrinsic cores of its faces. This result is tied up to the fact that any point of a convex set belongs to the intrinsic core of its minimal face and that the intersection of intrinsic cores of any pair of distinct faces is empty. Neither the quasi-relative interior nor the face relative interior satisfies these properties. In both cases, two distinct faces may have a nonempty intersection of their quasi- or face relative interiors. 

It may be possible however to obtain a disjoint decomposition of a convex set into the face relative interiors of what we call closure-equivalent faces of a convex set. Such decomposition is not available in terms of the quasi-relative interiors. 

\begin{definition}\label{defeqiv}
Given the convex sets $A,B\subseteq X$, where $X$ is a topological vector space, we say that $A$ is closure equivalent to $B$, denoted by $A \simeq B$, if $\overline{A}=\overline{B}$. The equivalence class of the face $F$ of a convex set $C$, is defined by $[F]^C:=\{A\unlhd C: \bar{A}=\bar{F}\}$. 
\end{definition}
\begin{proposition}\label{partition}
Let $C\subset X$ be a convex set in a topological vector space $X$. Let $A,B\subseteq C$ two faces of $C$, such that $A$ and $B$ are not closure equivalent, that is, $\overline{A}\neq \overline{B}$. Then $\fri A\cap \fri B=\emptyset$.  
\end{proposition}
\proof Let $A$ and $B$ be faces of $C$ such that $\bar A \neq \bar B$. Assume that there is an $x\in \fri A\cap \fri B$. This implies, by definition, that $x\in A\cap B$. Now, since $\overline{A}\neq \overline{B}$ there exist $y\in A\cup B$ such that $y\notin \overline{A}\cap \overline{B}$. Note that $\overline{F_{\min}(x,A\cap B)}\subseteq \overline{A\cap B}\subseteq \overline{A}\cap \overline{B}$, then $y\notin \overline{F_{\min}(x,A\cap B)}$. Suppose, without loss of generality, that $y\in A$, using that $F_{\min}(x,A\cap B)$ is a face of $A$ containing $x$ (since $A\cap B$ is a face of $C$ as the intersection of faces of $C$, it is also a face of $A$, likewise any face of $A\cap B$ is a face of $A$), then $F_{\min}(x,A)\subseteq F_{\min}(x,A\cap B)$, then $A\nsubseteq \overline{F_{\min}(x,A)}$  contradicting that $x\in \fri A$. This proves that our assumption was wrong and we must have $\fri A\cap \fri B = \emptyset$.
\endproof
 
Note that Proposition~\ref{partition} is not valid when the face relative interior is replaced by the quasi-relative interior, as we demonstrate in the next example. First, we recall a result about the intrinsic cores of minimal faces (for a proof see \cite[Corollary~3.14]{mill-rosh}).

\begin{proposition}\label{prop:icrminf} Let $C$ be a convex subset of a real vector space $X$, and let $F$ be a face of $C$. Then $F = F_{\min}(x,C)$ if and only if $x \in \icr F$.
\end{proposition} 

\begin{example}\label{nonpartition}
Suppose $C\subset l_2$ is like in Example \ref{exa1}, that is, $C = \{x\, |\, \|x\|_1 \leq 1\}$. Take $\bar{x}\in C$ such that $\|\bar{x}\|_1=1$, and such that $\bar x$ has an infinite number of nonzero entries, and assume for simplicity that all entries of $\bar x$ are nonnegative. We have demonstrated earlier (see Example~\ref{exa1}) that in this case 
\[
F_{\min }(\bar x,C) \subseteq C\cap \left\{x\,|\, \sum_{i=1}^\infty x_i =1, x_i \geq 0 \; \forall i \in \N\right\},
\] and that $\overline{F_{\min}(\bar x,C)}\nsupseteq C$. 
By Proposition~\ref{prop:icrminf} $\bar x\in \icr F_{\min}(x,C)$, also $\icr F_{\min}(x,C) \subseteq \qri F_{\min}(x,C)$ by Theorem~\ref{thm:interiorsandwich}. We hence have $\bar x \in \qri F_{\min}(x,C)$. At the same time, we know from the representation \eqref{eq:qriEx1} that $\bar x\in \qri C$. We conclude that $\bar{x}\in \qri F_{\min}(\bar{x},C)\cap \qri C$, hence the quasi-relative interiors of faces that aren't closure-equivalent may have nonempty intersections. 
\end{example}

We can therefore obtain the representation of a convex set as the disjoint union of the face relative interiors of its closure-equivalent faces, as we demonstrate next.

\begin{theorem}\label{thm:decompose} Let $C$ be a convex subset of a topological vector space $X$, and let $\mathcal F_C$ be the set of closure-equivalent classes of faces of $C$. Then 
\begin{equation}\label{eq:disjoint}
C  = \dbigcup_{[F]^C\in \mathcal F_C}  \bigcup_{E\in [F]^C}\fri E,
\end{equation}
where by dot over the union sign we denote that the union is disjoint.    
\end{theorem}
\begin{proof}
Since for any $x\in C$ we have $x\in \icr F_{\min}(x,C)$ by Proposition~\ref{prop:icrminf}, and $\icr F_{\min}(x,C) \subseteq \qri F_{\min}(x,C)$ by Theorem~\ref{thm:interiorsandwich}, we conclude that 
\[
C  \subseteq \bigcup_{F\unlhd C} \icr F \subseteq \bigcup_{F\unlhd C} \fri F  = \bigcup_{[F]^C\in \mathcal F_C}  \bigcup_{E\in [F]^C}\fri E.
\]
Moreover, since for any face $F$ of $C$ we have $\fri F \subseteq F\subset C$, we sharpen this inclusion to 
\begin{equation}\label{eq:uniontechnical}
C  = \bigcup_{F\unlhd C} \fri F  = \bigcup_{[F]^C\in \mathcal F_C}  \bigcup_{E\in [F]^C}\fri E.
\end{equation}
Finally, from Proposition~\ref{partition} we conclude that the union in \eqref{eq:uniontechnical} is disjoint, arriving at \eqref{eq:disjoint}.
\end{proof}

\begin{remark} Note that due to Proposition~\ref{prop:closures} we have for any closure-equivalent class $[F]$ of faces of some convex set $C$ that 
\[
\bigcup_{E\in [F]}\fri E = \fri \overline{F} \cap \bigcup_{E\in [F]} E, 
\]
giving an alternative representation for the sets forming the disjoint partition in \eqref{eq:disjoint}.
\end{remark}

Recall that any convex set $C$ in a real vector space $X$ decomposes into a disjoint union of the intrinsic cores of its faces. This decomposition may coincide with the one in Theorem~\ref{thm:decompose} (for instance, in the finite-dimensional setting, when all interior notions coincide), but generally speaking these two decompositions are different. To see how they compare in the case when the intrinsic cores don't match the face relative interiors, we revisit Examples~\ref{eg:02a} and~\ref{eg:icrneqfri}. The former is a more intricate case than the latter, so we start with a simple case.

\begin{example} Let $C$ be as in Example~\ref{eg:icrneqfri}. There are only three minimal faces: $\{u\}$, $L$, and the entire set $C$. We have 
\[
C = \icr \{u\}\cup \icr L \cup \icr C = \{u\}\cup L \cup (C\setminus (L\cup \{u\})).
\]
At the same time since the only point in $C$ for which $C\nsubseteq \overline{F_{\min}(x,C)}$ is $\{u\}$, we conclude that the decomposition of Theorem~\ref{thm:decompose} gives 
\[
C = \{u\}\cup (C\setminus \{u\}).
\]
\end{example}

\begin{example} Let $C$ be as in Example~\ref{eg:02a}. We have determined that 
\[
\fri C = \{x\in C \,|\, \|x\|_2 <1, \; x_i>0\; \forall\, i \in \N\}.
\]
Observe that for any nonempty $N\subseteq \N$ the set 
\[
F_N = \{x\in l_2\,|\, \|x\|_2\leq 1, x_i \geq 0 \forall i \in \N\setminus N, \, x_i = 0 \forall i \in N\}
\]
is a face of $C$, and using the same argument based on truncated sequences as in the discussion of Example~\ref{eg:02a}, we conclude that for any nonempty $N\subseteq \N$
\[
\fri F_N = \{x\in l_2\,|\, \|x\|_2< 1, x_i > 0 \;\forall i \in \N\setminus N, \, x_i = 0 \forall i \in N\}. 
\]
The only remaining points are the ones of unit norm, and we have already determined that in this case $\fri F_{\min}(x,C) = F_{\min}(x,C) = \{x\}$. We hence have a decomposition of $C$ as follows,
\[
C  = \bigcup_{x\in C, \|x\|_2 =1} \{x\} \cup \bigcup_{N\subseteq \N}  \{x\in l_2\,|\, \|x\|_2< 1, x_i > 0 \, \forall i \in \N\setminus N, \, x_i = 0 \, \forall i \in N\}.
\]
Recall how in this case $\icr C = \emptyset$, and so $C$ is never the minimal face of $x\in C$. Based on the observation that minimal faces are determined by the comparative speed with which the sequences that define the points in $C$ converge, the respectful decomposition of the set $C$ via the intrinsic cores is different to the one that we have just obtained via the face relative interiors. 
\end{example}

\section{Conclusions}

We have introduced a new notion of face relative interior for convex sets in topological vector spaces. Face relative interior can be seen as an alternative to the quasi-relative interior that is better aligned with the facial structure of convex sets. 

Even though we have placed the new notion in the context of other research with the sandwich theorem and examples, and provided an extensive discussion on calculus rules, there is still much to be understood. For instance, due to Theorem~\ref{thm:interiorsandwich} we always have $\fri C \subseteq \qri C$, also in separable Banach spaces, both $\fri C$ and $\qri C$ are nonempty for any convex set $C$. However, we don't know if there exists a convex set in some topological vector space for which the face relative interior is empty and the quasi-relative interior isn't. Another challenge is that the intersection property given in Proposition~\ref{prop:intersection} is only proved under an unusual assumption that the face relative interior of the intersection must coincide with the intrinsic core, and since we know that without this assumption the result may fail (as seen in Example~\ref{eg:intersectionfails}), we are wondering what is the weakest condition that guarantees this property to hold.

We have not touched on the topics of duality within this work, which is an exciting subject for future research. We anticipate that face-relative interior will find practical use whenever the facial structure of a convex set is easy to access, however, it remains to be seen if any interesting applications emerge in convex geometry and optimisation.

\section{Acknowledgements}

We are grateful to Prof. David Yost for pointing out some relevant references to us and for productive discussions in the initial stages of working on this project, and to the anonymous referees who contributed valuable expertise as well as a substantial amount of their time reviewing this work, and whose insightful comments and feedback have greatly improved our paper. 

We also thank the Australian Research Council for the financial support provided by means of Discovery Projects ``An optimisation-based framework for non-classical Chebyshev approximation'', DP180100602 and ``Geometry in projection methods and fixed-point theory'' DP200100124.

\bibliographystyle{plain}

\bibliography{references}

\end{document}